\documentclass[12pt]{article}
\usepackage[a4paper,margin=1in]{geometry}
\usepackage{amssymb}
\usepackage{amsmath}
\usepackage{amsfonts}
\usepackage{amsthm}
\usepackage{color}
\usepackage{float}
\usepackage[all]{xy}
\usepackage{scalerel}
\usepackage{mathabx}
\usepackage{delimset}

\newcommand{\C}{\mathbb{C}}

\newcommand{\R}{\mathbb{R}}

\newcommand{\e}{\mathrm{e}}

\newcommand{\rr}{\mathrm{R}}

\newcommand{\pp}{\mathrm{P}}
\newcommand{\E}{\mathrm{E}}

\newcommand{\RR}{\mathcal{R}}

\newcommand{\qbinom}{\genfrac{[}{]}{0pt}{}}

\newtheorem{theorem}{Theorem}
\newtheorem{lemma}{Lemma}

\newtheorem{proposition}{Proposition}

\begin{document}

\title{\textbf{A Rigidity Property of the Deformed q-Exponential Function}}
\author{Ronald Orozco L\'opez}

\newcommand{\Addresses}{{
  \bigskip
  \footnotesize

  \textit{E-mail address}, R.~Orozco: \texttt{rj.orozco@uniandes.edu.co}
  
}}

\maketitle

\begin{abstract}
The deformed \(q\)-exponential function \[ e_q(z,u) = \sum_{n=0}^{\infty} u^{\binom{n}{2}} \frac{z^n}{(q;q)_n} \] provides a common framework containing several classical \(q\)-exponential functions as particular cases, including the Jackson \(q\)-exponentials \(e_q(z)\) and \(E_q(z)\). In this paper we investigate the multiplicative inversion problem \[ e_q(z,u)e_q(-z,v)=1, \] and determine all pairs of deformation parameters \((u,v)\) for which this identity holds. To this end, we introduce a family of coefficient polynomials whose common zeros characterize the inversion property. A geometric analysis of the first nontrivial coefficients reduces the problem to two parameter branches. The symmetric branch is excluded through a parity phenomenon, while the nonsymmetric branch is completely determined by the first two coefficient constraints. As a consequence, we prove a rigidity theorem showing that \[ e_q(z,u)e_q(-z,v)=1 \] if and only if \[ (u,v)=(1,q) \qquad\text{or}\qquad (u,v)=(q,1). \] Thus the classical Jackson inversion identity is rigid within the deformed family and no new multiplicative inversion identities arise from the deformation parameter.
\end{abstract} 
{\bf Keywords:} Deformed $q$-exponential function, Jackson $q$-exponentials, multiplicative inversion, rigidity theorem, $q$-special functions.\\
{\bf Mathematics Subject Classification:} 33D15, 05A30, 33D05.

\section{Introduction} 

One of the most fundamental identities in the theory of \(q\)-special functions is the classical Jackson inversion formula \[ e_q(z)E_q(-z)=1, \] where \[ e_q(z) = \sum_{n=0}^{\infty} \frac{z^n}{(q;q)_n} \] and \[ E_q(z) = \sum_{n=0}^{\infty} q^{\binom{n}{2}} \frac{z^n}{(q;q)_n}. \] This identity plays a central role in \(q\)-analysis and reflects the complementary nature of the two Jackson \(q\)-exponential functions.

Recently, the deformed \(q\)-exponential function \[ e_q(z,u) = \sum_{n=0}^{\infty} u^{\binom{n}{2}} \frac{z^n}{(q;q)_n} \] was introduced as a unifying framework for several classical \(q\)-special functions. Indeed, \[ e_q(z,1)=e_q(z), \qquad e_q(z,q)=E_q(z), \] while other choices of the deformation parameter lead to functions related to the Exton and Rogers--Ramanujan \(q\)-exponentials. This naturally raises the question of which classical properties of \(q\)-analysis survive under deformation.

The present paper focuses on one of the most basic structural properties of the Jackson \(q\)-exponentials, namely the inversion identity. Since the classical formula can be written as \[ e_q(z,1)e_q(-z,q)=1, \] it is natural to ask whether other values of the deformation parameters produce similar identities. More precisely, we investigate the multiplicative inversion problem \[ e_q(z,u)e_q(-z,v)=1. \]

Our approach is based on the introduction of a family of coefficient polynomials \(P_n(u,v;q)\) obtained from deformed homogeneous polynomials associated with the product \(e_q(z,u)e_q(-z,v)\). The inversion problem is thereby converted into an algebraic problem concerning the common zeros of the sequence \(\{P_n(u,v;q)\}_{n\ge1}\). A geometric analysis of the first coefficients reveals a remarkable splitting of the parameter space into symmetric and nonsymmetric branches.

The main result of the paper is a rigidity theorem showing that the classical Jackson identity is isolated within the deformed family. More precisely, we prove that \[ e_q(z,u)e_q(-z,v)=1 \] holds only for the classical pairs \[ (u,v)=(1,q) \] and \[ (u,v)=(q,1). \] In particular, no new multiplicative inversion identities arise from the deformation parameter.

\section{Preliminaries} 

This section collects the notation and basic facts from \(q\)-analysis that will be used throughout the paper. We assume \(0<q<1\).

\subsection{$q$-shifted factorials} 

A basic object of the theory is the $q$-shifted factorial, defined by \[ (a;q)_n= \begin{cases} 1, & n=0,\\[2mm] \displaystyle\prod_{k=0}^{n-1}(1-aq^k), & n\ge1, \end{cases} \qquad a\in\mathbb{C}. \] The infinite $q$-shifted factorial is \[ (a;q)_\infty = \prod_{k=0}^{\infty}(1-aq^k), \qquad |q|<1. \] 

\subsection{The Deformed $q$-Exponential Function} 

In \cite{orozco} was defined the deformed $q$-exponential by
\begin{equation}\label{eqn_dqexp}
\e_q(z,u)=\sum_{n=0}^{\infty}u^{\binom{n}{2}}\frac{z^n}{(q;q)_n}.
\end{equation}
From Eq.(\ref{eqn_dqexp}), some deformed $q$-exponential functions are:
The Euler $q$-exponential functions are obtained by set $u=1$ and $u=q$, respectively. They are
    \[
    e_{q}(z)=\e_{q}(z,1)=\sum_{n=0}^{\infty}\frac{z^n}{(q;q)_n}=\frac{1}{(z;q)_\infty},
    \]
    convergent in $\vert z\vert<\frac{1}{1-q}$, and
    \[
    \E_{q}(z)=\e_{q}(z,q)=\sum_{n=0}^{\infty}q^{\binom{n}{2}}\frac{z^n}{(q;q)_n}=(-z;q)_\infty
    \]
    convergent for all $z\in\C$. 
If $u=\sqrt{q}$ in Eq.(\ref{eqn_dqexp}), we obtain the Exton $q$-exponential function \cite{exton}
    \[
    \mathcal{E}_{q}(z)=\e_{q}(z,\sqrt{q})=\sum_{n=0}^{\infty}q^{\frac{1}{2}\binom{n}{2}}\frac{z^n}{(q;q)_n}={}_{1}\phi_{1}\left(\begin{array}{c}
         0\\
         -\sqrt{q}
    \end{array};\sqrt{q},-z\right),\ z\in\C.
    \]
By the mapping $z\mapsto qz$ and by set $u=q^2$ in Eq.(\ref{eqn_dqexp}), we obtain the Rogers-Ramanujan function
    \[
    \mathcal{R}_{q}(z)=\e_{q}(qz,q^2)=\sum_{n=0}^{\infty}q^{n^2}\frac{z^n}{(q;q)_n},\qquad z\in\C.
    \]
A representation of $\RR_q(z)$ in basic hypergeometric series is \cite{gasper}    
    \[
    \RR_q(z)=\frac{(zq^5,z^2q^2,z^2q^3;q^5)_{\infty}}{(zq;q)_{\infty}}{}_{3}\phi_{2}\left(\begin{array}{c}
         z/q,z,zq\\
         z^2q^2,z^2q^3
    \end{array};q^5,zq^5\right).    
    \]
More generally, for $A,B\in\R$, $A\geq0$, if $u=q^{2A}$ and by set $z\mapsto q^{A+B}z$, we obtain the generalized Rogers-Ramanujan function \cite{bailey,slater}
\[
\RR_{q}^{(A,B)}(z)=\e_{q}(q^{A+B}z,q^{2A})=\sum_{n=0}^{\infty}q^{An^2+Bn}\frac{z^n}{\brk[s]{n}_q!}.
\]

The examples above illustrate the breadth of the family \(\e_q(z,u)\). Since both Jackson \(q\)-exponentials arise as particular cases, it is natural to investigate whether the classical inversion identity admits a nontrivial extension within this larger deformed family. This question motivates the analysis developed in the next sections.

\section{Deformed homogeneous polynomials}

\subsection{Definition}

The multiplicative inversion problem studied in this paper naturally leads to a family of deformed homogeneous polynomials introduced in \cite{orozco}. These polynomials provide the coefficient structure of the product of two deformed $q$-exponential functions and therefore serve as the appropriate framework for our analysis. We define for $u,v\in\mathbb C$ the $(u,v)$-deformed homogeneous polynomials
\begin{equation}
    \rr_{n}(x,y;u,v|q)=\sum_{k=0}^{n}\qbinom{n}{k}_{q}u^{\binom{n-k}{2}}v^{\binom{k}{2}}x^{n-k}y^{k}.
\end{equation}
The polynomials $\rr_n$ has the following representation
\begin{equation}\label{eqn_repre}
    \e_q(xz,u)\e_q(yz,v)=\sum_{n=0}^{\infty}\rr_n(x,y;u,v\,|\,q)\frac{z^n}{(q;q)_n}.
\end{equation}

\subsection{The coefficient polynomials} 

The inversion problem becomes considerably more transparent after a specialization of the homogeneous basis. For this reason we introduce the coefficient polynomials
\begin{equation} 
\pp_n(u,v;q) := \rr_n(1,-1;u,v\,|\,q). 
\end{equation}
Then
\begin{equation}\label{eqn_PolyUV}
\pp_n(u,v;q) = \sum_{k=0}^{n} (-1)^k \qbinom{n}{k}_q u^{\binom{n-k}{2}} v^{\binom{k}{2}} . 
\end{equation}

\subsection{Generating function}

The relevance of the polynomials \(\pp_n(u,v;q)\) stems from the fact that they appear as the coefficients of the product \(\e_q(z,u)\e_q(-z,v)\). Consequently, the multiplicative inversion problem can be translated into a problem about the simultaneous vanishing of a sequence of coefficient polynomials. From Eq.(\ref{eqn_repre}), the representation of the polynomials $\pp_n$ is
\begin{equation}\label{eqn_GFP}
    \e_q(z,u)\e_q(-z,v) = \sum_{n=0}^{\infty} \pp_n(u,v;q) \frac{z^n}{(q;q)_n}.
\end{equation}

\begin{proposition}\label{prop_AI-AP}
    \[ \e_q(z,u)\e_q(-z,v)=1 \] 
    if and only if 
    \[ \pp_n(u,v;q)=0, \qquad n\ge 1. \]
\end{proposition}

Proposition~1 transforms the analytic identity 
\[ \e_q(z,u)\e_q(-z,v)=1 \] into an algebraic problem. From this point on, our task is to determine all parameter pairs \((u,v)\) for which every coefficient polynomial \(\pp_n(u,v;q)\) vanishes.

\subsection{Basic coefficients}

\begin{proposition}
   \[ \pp_{2}(u,v;q) = u+v-(1+q). \] 
\end{proposition}

\begin{proposition}
    \[ \pp_3(u,v;q) = (u-v) \Bigl( u^2+uv+v^2-(1+q+q^2) \Bigr). \]
\end{proposition}
\begin{proof}
Direct computation from (\ref{eqn_PolyUV}) gives 
\[ 
    \pp_3(u,v;q) = u^3-(1+q+q^2)u+ (1+q+q^2)v-v^3. 
\] 
Factoring yields the result.
\end{proof}

The first two nontrivial coefficients already contain a remarkable amount of information. As we shall see, the conditions \[\pp_2(u,v;q)=0, \qquad \pp_3(u,v;q)=0 \] force the parameter space to split into only two possible branches. This geometric splitting is described in the next result.

\begin{lemma}[\textbf{Geometry of the first constraints}]\label{lemma_geo} 
The simultaneous conditions \[ \pp_2(u,v;q)=0, \qquad \pp_3(u,v;q)=0 \] define the intersection of the affine line \[ u+v=1+q \] with the cubic variety 
\[ 
(u-v)\Bigl(u^2+uv+v^2-(1+q+q^2)\Bigr)=0. 
\] 
Consequently, every solution belongs to one of the following two branches: \[ u=v, \] or \[ u^2+uv+v^2=1+q+q^2. \]
\end{lemma}
\begin{proof}
From Proposition 2, 
\[ \pp_2(u,v;q)=0 \] 
is equivalent to 
\[ u+v=1+q. \] 
By Proposition 3, 
\[ \pp_3(u,v;q)=0 \] 
if and only if 
\[ (u-v) \Bigl( u^2+uv+v^2-(1+q+q^2) \Bigr)=0. \] 
Therefore every common solution of \(\pp_2(u,v;q)=\pp_3(u,v;q)=0\) must lie on the affine line 
\[ u+v=1+q, \] 
and simultaneously belong to one of the two branches 
\[ u=v, \] 
or 
\[ u^2+uv+v^2=1+q+q^2. \] 
This completes the proof.
\end{proof}
Lemma \ref{lemma_geo} reduces the problem to two distinct parameter branches. We first investigate the symmetric branch \(u=v\). Rather unexpectedly, this case exhibits a parity phenomenon.

\section{A Rigidity Theorem for the Deformed $q$-Exponentials} 

\subsection{The branch $u=v$}

When the deformation parameters coincide, the coefficient polynomials acquire a strong symmetry that forces all odd coefficients to vanish.

\begin{proposition}\label{prop_OCV}
  For every $n\ge0$, \[ \pp_{2n+1}(u,u;q)=0. \]  
\end{proposition}
\begin{proof}
From Definition (\ref{eqn_PolyUV}), 
\[ 
\pp_{2n+1}(u,u;q) = \sum_{k=0}^{2n+1} (-1)^k \qbinom{2n+1}{k}_q u^{\binom{2n+1-k}{2}+\binom{k}{2}}. 
\] 
Let 
\[ 
T_k = (-1)^k \qbinom{2n+1}{k}_q u^{\binom{2n+1-k}{2}+\binom{k}{2}}. 
\] 
Using the symmetry of the $q$-binomial coefficients, 
\[ 
\qbinom{2n+1}{2n+1-k}_q = \qbinom{2n+1}{k}_q, 
\] 
we obtain 
\[ T_{2n+1-k} = (-1)^{2n+1-k} \qbinom{2n+1}{k}_q u^{\binom{k}{2} + \binom{2n+1-k}{2}}=-T_k. 
\] 
The terms cancel pairwise under the involution \[ k\longmapsto 2n+1-k. \] Because \(2n+1\) is odd, there is no fixed point of this involution. Hence every term cancels with a unique partner and 
\[ \pp_{2n+1}(u,u;q)=0. \] The proof is complete.
\end{proof}
The previous proposition shows that all odd coefficients in the expansion 
\[ 
\e_q(z,u)\e_q(-z,u) = \sum_{n=0}^{\infty} \pp_n(u,u;q) \frac{z^n}{(q;q)_n} 
\] 
vanish identically. Consequently, 
\[ 
\e_q(z,u)\e_q(-z,u) = 1+ \sum_{n=1}^{\infty} \pp_{2n}(u,u;q) \frac{z^{2n}}{(q;q)_{2n}}. 
\]
Thus, in the symmetric case \(u=v\), the product \(\e_q(z,u)\e_q(-z,u)\) is an even formal power series. The parity phenomenon might suggest the possibility of additional inversion identities within the symmetric branch. The next result shows that this is not the case. 

\begin{lemma}\label{lemma_SENS}
Let \(0<q<1\). Then the system \[ 
\pp_2(u,u;q)=0, \qquad \pp_4(u,u;q)=0 
\] 
has no solution.
\end{lemma}
\begin{proof}
From \(\pp_2(u,u;q)=0\) we obtain 
\[ 
u=\frac{1+q}{2}. 
\] 
Substituting into \(\pp_4\) yields 
\[ 
\pp_4\!\left(\frac{1+q}{2},\frac{1+q}{2};q\right) = \frac{(1+q)^2}{32} \Bigl( 5+4q+14q^2+4q^3+5q^4 \Bigr). 
\] 
Since all coefficients are positive and \(0<q<1\), the right-hand side is strictly positive. Hence 
\[ \pp_4(u,u;q)\neq0. \] 
Therefore, the system has no solution.
\end{proof}
Lemma~2 rules out the symmetric branch entirely. Consequently, any possible multiplicative inversion must arise from the nonsymmetric branch identified in Lemma~1. The remainder of the section is devoted to showing that this branch consists precisely of the classical pairs \((1,q)\) and \((q,1)\).

\subsection{The branch $u\neq v$}

By Lemma~\ref{lemma_SENS}, the symmetric branch cannot produce a multiplicative inversion identity. Consequently, it remains only to analyze the nonsymmetric branch arising from Lemma~\ref{lemma_geo}. Remarkably, the first two coefficient constraints already determine all admissible parameter pairs.

\begin{lemma}\label{lemma_SP23}
If \[ \pp_2(u,v;q)=P_3(u,v;q)=0 \] and $u\neq v$, then \[ (u,v)=(1,q) \] or \[ (u,v)=(q,1). \] 
\end{lemma} 
\begin{proof}
From $\pp_2(u,v;q)=0$ we obtain
\[
u+v=1+q.
\]
As $u\neq v$, then from Lemma \ref{lemma_geo}
\[
u^2+uv+v^2=1+q+q^2.
\]
As
\[
u^2+uv+v^2=(u+v)^2-uv,
\]
we obtain
\[
(1+q)^2-uv=1+q+q^2.
\]
Then
\[
uv=q,\qquad u+v=1+q.
\]
Therefore $(u,v)\in\{(1,q),(q,1)\}$.
\end{proof}
Hence the nonsymmetric branch consists exactly of the two classical Jackson pairs.

\begin{theorem}\label{theo_rigidez}
The identity \[ \e_q(z,u)\e_q(-z,v)=1 \] holds if and only if 
\[ (u,v)\in\{(1,q),(q,1)\}. \]
\end{theorem} 
\begin{proof}
By Proposition~\ref{prop_AI-AP}, the identity 
\[ \e_q(z,u)\e_q(-z,v)=1 \] 
is equivalent to the vanishing of all coefficient polynomials 
\[ \pp_n(u,v;q), \qquad n\ge1. \] 
In particular, 
\[ \pp_2(u,v;q)=\pp_3(u,v;q)=0. \] 
By Lemma~\ref{lemma_geo}, every solution belongs either to the symmetric branch \(u=v\) or the nonsymmetric branch \(u\neq v\). Lemma~\ref{lemma_SENS} excludes the symmetric branch. Hence \(u\neq v\). Applying Lemma~\ref{lemma_SP23} yields 
\[ (u,v)=(1,q) \] 
or 
\[ (u,v)=(q,1). \] 
Conversely, the classical Jackson identity 
\[ \e_q(z)\E_q(-z)=1 \] 
implies 
\[ \e_q(z,1)\e_q(-z,q)=1. \] 
Interchanging the two factors gives 
\[ \e_q(z,q)\e_q(-z,1)=1. \] 
Therefore the stated pairs are precisely the solutions of the multiplicative inversion problem.
\end{proof}
Theorem \ref{theo_rigidez} shows that the classical Jackson identity is rigid within the deformed family \(\e_q(z,u)\). No new multiplicative inversion identities arise from the deformation parameter.

\section{Conclusion and Future directions}

In this paper we studied the multiplicative inversion problem for the deformed \(q\)-exponential family \(\e_q(z,u)\). By introducing a sequence of coefficient polynomials and analyzing the first nontrivial coefficient constraints, we reduced the problem to a finite algebraic classification. A geometric decomposition of the parameter space revealed two possible branches. The symmetric branch was excluded through a parity argument, while the nonsymmetric branch was shown to contain only the classical Jackson pairs. As a consequence, we proved that 
\[ \e_q(z,u)\e_q(-z,v)=1 \] 
if and only if 
\[ (u,v)=(1,q) \] 
or 
\[ (u,v)=(q,1). \] 
Thus the classical Jackson inversion identity is rigid within the deformed family. 

The result suggests a broader research direction. While the family \(\e_q(z,u)\) provides a rich deformation of classical \(q\)-special functions, structural identities need not survive under deformation. It would therefore be interesting to investigate whether similar rigidity phenomena occur for \(q\)-trigonometric identities, addition formulas, \(q\)-difference equations, operational representations and other characteristic properties of the classical \(q\)-functions. We hope that the methods developed here may serve as a first step towards a systematic study of rigidity phenomena in deformed \(q\)-special functions.

\section{Statements and Declarations}
\subsection{Conflict of Interests}
We have no conflict of interest to disclose.

\subsection{Data availability}
The author confirms that the manuscript does not use known data.

\subsection{Declaration of non-funding}
The authors received no support from any organization for the submitted work.

\bibliographystyle{plain}

\end{document}